\documentclass[11pt,twoside,leqno,a4paper]{amsart}

\usepackage{amssymb}
\usepackage{enumerate}
\newtheorem{thm}{Theorem}
\newtheorem{cor}[thm]{Corollary}
\newtheorem{lem}[thm]{Lemma}
\newtheorem{prop}[thm]{Proposition}
\theoremstyle{definition}
\newtheorem{defn}{Definition}
\theoremstyle{remark}
\newtheorem{rem}{Remark}
\numberwithin{equation}{section}
\numberwithin{thm}{section}
\DeclareMathOperator{\V}{V}
\DeclareMathOperator{\D}{d}
\newcommand{\e}{e}

\DeclareMathOperator{\h}{\chi}

\newcommand{\abs}[1]{\left\vert#1\right\vert}
\newcommand{\set}[1]{\left\{#1\right\}}

\everymath{\displaystyle}
\begin{document}

\title[$\delta_k$-small sets in graphs]
 {$\delta_k$-small sets in graphs}

\author{Asen Bojilov}

\address{Faculty of Mathematics and Informatics,
University of Sofia, Bulgaria}

\email{bojilov@fmi.uni-sofia.bg}

\author{Nedyalko Nenov}

\address{Faculty of Mathematics and Informatics,
University of Sofia, Bulgaria}

\email{nenov@fmi.uni-sofia.bg}

\thanks{This work was supported by the Scientific Research Fund
of the St. Kliment Ohridski Sofia University 2013.}



%\date{February 15, 1995 and, in revised form, July 6, 1995.}

\dedicatory{}

%\commby{Daniel J. Rudolph}

% -----------------------------------------------------------

\begin{abstract}
Let $G$ be a simple $n$-vertex graph and $W\subseteq\V(G)$. We say that
$W$ is a $\delta_k$-small set if
$$
\sqrt[k]{\frac{\sum_{v\in W}d^k(v)}{\abs W}}\leq n-\abs W.
$$
Let $\varphi^{(k)}(G)$ denote the smallest natural number $r$ such that
$\V(G)$ decomposes into $r$ $\delta_k$-small sets, and let $\alpha^{(k)}(G)$
denote the maximal number of vertices in a $\delta_k$-small set of $G$.
In this paper we obtain bounds for $\alpha^{(k)}(G)$ and $\varphi^{(k)}(G)$.
Since $\varphi^{(k)}(G)\leq\omega(G)\leq\chi(G)$ and $\alpha(G)\leq\alpha^{(k)}(G)$,
we obtain also bounds for the clique number $\omega(G)$, the chromatic number $\chi(G)$ and
the independence number $\alpha(G)$.
\end{abstract}

% -----------------------------------------------------------
\maketitle
% -----------------------------------------------------------

\section[Introduction]{Introduction}

We consider only finite, non-oriented graphs without loops and
multiple edges. We shall use the following notations:

$\V(G)$ -- the vertex set of $G$;

$\e(G)$ -- the number of edges of $G$;

$\omega(G)$ -- the clique number of $G$;

$\chi(G)$ -- the chromatic number of $G$;

$d(v)$ -- the degree of a vertex $v$;

$\Delta(G)$ -- the maximal degree of $G$;

$\delta(G)$ -- the minimal degree of $G$.

All undefined notation are from \cite{west01}.

\begin{defn}\label{defn1.1}
Let $G$ be an $n$-vertex graph and $W\subseteq \V(G)$. We say that
$W$ is a \emph{small set} in the graph $G$ if
$$
d(v)\leq n-\abs{W},\text{ for all }v\in W.
$$
With $\varphi(G)$ we denote the smallest natural number $r$ such that
$\V(G)$ decomposes into $r$ small sets.
\end{defn}

$\varphi(G)$ is defined for the first time in \cite{nen06}. Some properties of
$\varphi(G)$ are proved in \cite{nen06} and \cite{bn12}.
Further $\varphi(G)$ is more thoroughly investigated in
\cite{bcan}. There an effective algorithm for the
calculation of $\varphi(G)$ is given. First of all let us note the following bounds for $\varphi(G)$.

\begin{prop}[\cite{bcan}]\label{prop1.1}
$$
\left\lceil\frac n{n-\D_1(G)}\right\rceil\leq\varphi(G)\leq
\left\lceil\frac n{n-\Delta(G)}\right\rceil,
$$
where $\D_1(G)$ is the average degree of the graph $G$.
\end{prop}

Let $G$ be a graph and $W\subseteq\V(G)$. We define
$$
\D_k(W)=\sqrt[k]{\frac{\sum\limits_{v\in W}d^k(v)}{\abs W}},\qquad
\D_k(G)=\D_k\big(\V(G)\big).
$$

\begin{defn}\label{defn1.2}
Let $G$ be an $n$-vertex graph and $W\subseteq\V(G)$. We say that $W$ is
a \emph{$\delta_k$-small set} of $G$ if
$$
\D_k(W)\leq n-\abs W.
$$
With $\varphi^{(k)}(G)$ we denote the minimal number of $\delta_k$-sets of $G$ into which
$\V(G)$ decomposes.
\end{defn}

\begin{rem}\label{rem1.1}
$\delta_1$-small sets are defined in \cite{bcan} as $\beta$-small sets and
$\varphi^{(1)}(G)$ is denoted by $\varphi^\beta(G)$. Also in \cite{bcan}
it is proven
\end{rem}

\begin{prop}[\cite{bcan}]\label{prop1.2}
$$
\varphi^{(1)}(G)\geq\left\lceil\frac n{n-\D_1(G)}\right\rceil.
$$
\end{prop}

Further we shall need the following

\begin{prop}\label{prop1.3}
Let $G$ be an $n$-vertex graph. Then
\begin{enumerate}[{\rm (i)}]
\item
Every small set of $G$ is a $\delta_k$-small set of $G$ for all natural $k$.
\item
Every $\delta_{k-1}$-small set of $G$ is a $\delta_k$-small set of $G$.
\end{enumerate}
\end{prop}

\begin{proof}
Let $W$ be a small set of $G$. Then $d(v)\leq n-\abs W$, $\forall v\in W$. Therefore
$\D_k(W)\leq n-\abs W$, i.\,e. $W$ is a $\delta_k$-small set.

The statement in (ii) follows from the inequality
$\D_{k-1}(W)\leq\D_k(W)$ (cf. \cite{hl34,kh90}).
\end{proof}

Let us note that if $G$ is an $r$-regular graph then
$\D_k(W)=r$ for all natural $k$. So, in this case,
every $\delta_k$-set of $G$ is a small set of $G$.

In this paper we shall prove that for a given graph $G$ and for
sufficiently large natural $k$ every $\delta_k$-small set of $G$
is a small set of $G$ (Theorem~\ref{thm2.1}).

\begin{prop}\label{prop1.4}
Let $G$ be a graph. Then
$$
\varphi^{(1)}(G)\leq\varphi^{(2)}(G)\leq\cdots\leq\varphi^{(k)}(G)
\leq\cdots\leq\varphi(G)\leq\omega(G)\leq\chi(G).
$$
\end{prop}

\begin{proof}
The inequality $\chi(G)\geq\omega(G)$ is obvious. The inequality
$\varphi(G)\leq\omega(G)$ is proven in \cite{nen06} (see also \cite{bcan}).
The inequality $\varphi^{(k)}(G)\leq\varphi(G)$ follows from
Proposition~\ref{prop1.3}~{(i)} and the inequlity
$\varphi^{(k-1)}(G)\leq\varphi^{(k)}(G)$ follows from
Proposition~\ref{prop1.3}~{(ii)}.
\end{proof}

According to Proposition \ref{prop1.4} every lower bound for
$\varphi^{(k)}(G)$ is a lower bound for $\varphi(G)$, $\omega(G)$
and $\h(G)$. In this paper we shall obtain a lower bound for
$\varphi^{(k)}(G)$ (Theorem~\ref{thm3.2}) from which we shall derive
new lower bounds for $\varphi(G)$, $\omega(G)$ and $\h(G)$.
As a corollary we shall get and some results
for $\varphi(G)$, $\omega(G)$ and $\chi(G)$ already from \cite{bcan} and \cite{bn12}.

%At the end of this section we shall prove the following

\begin{prop}\label{prop1.5}
$$
\left\lceil\frac n{n-\D_1(G)}\right\rceil\leq\varphi^{(k)}(G)\leq
\left\lceil\frac n{n-\Delta(G)}\right\rceil.
$$
\end{prop}

\begin{proof}
The right inequality follows from Proposition~\ref{prop1.1} and
Proposition~\ref{prop1.4}. The left inequality follows from
Proposition~\ref{prop1.2} and Proposition~\ref{prop1.4}.
\end{proof}

% -----------------------------------------------------------

\section[Strengthening Proposition \ref{prop1.4}]{Strengthening Proposition \ref{prop1.4}}

\begin{thm}\label{thm2.1}
Let $G$ be a graph. There exists a natural $k_0=k_0(G)$ such that for all
$k\geq k_0$ we have
\begin{enumerate}[{\rm (i)}]
\item
Every $\delta_k$-small set of $G$ is a small set of $G$.
\item
$
\varphi^{(1)}(G)\leq\cdots\leq\varphi^{(k_0)}(G)=\varphi^{(k_0+1)}(G)=
\cdots=\varphi(G).
$
\end{enumerate}
\end{thm}

\begin{proof}
Fix a subset of $\V(G)$, say $W$, and let
$\Delta(W)=\max\set{d(v)\mid v\in W}$. Then
$\D_k(W)\leq\Delta(W)$ and $\lim_{k\to\infty}\D_k(W)=\Delta(W)$ (see \cite{hl34}).

Therefore, since $\V(G)$ has only finitely many subsets, there exists $k_0$ such that for arbitrary $W\subseteq\V(G)$
\begin{equation}\label{eq2.1}
\Delta(W)-\frac12\leq\D_k(W),\text{ if }k\geq k_0.
\end{equation}

Let us suppose now that $W$ is a $\delta_k$-small set of $G$ and $k\geq k_0$, i.\,e.
\begin{equation}\label{eq2.2}
\D_k(W)\leq n-\abs W.
\end{equation}
From \eqref{eq2.1} and \eqref{eq2.2} we have that
$$
\Delta(W)-\frac12\leq n-\abs W.
$$
Since $\Delta(W)$ and $n-\abs W$ are integers, from the last inequality we derive
that $\Delta(W)\leq n-\abs W$. From the definition of $\Delta(W)$ it follows
$d(v)\leq n-\abs W$ for all $v\in W$, i.\,e. $W$ is a small set.
Thereby (i) is proven. The statement (ii) obviously follows from (i).
\end{proof}

% -----------------------------------------------------------

\section[Lower bounds for $\D_k(G)$ and $\varphi^{(k)}(G)$]
{Lower bounds for $\D_k(G)$ and $\varphi^{(k)}(G)$}

\begin{lem}\label{lm3.1}
Let $\beta_1,\beta_2,\dots,\beta_r\in[0,1]$ and
$\beta_1+\beta_2+\cdots+\beta_r=r-1$. Then for all natural
$k\leq r$ is held the inequality
\begin{equation}\label{eq3.1}
\sum_{i=1}^r(1-\beta_i)\beta_i^k\leq\left(\frac{r-1}r\right)^k.
\end{equation}
\end{lem}

\begin{proof}
The case $k=r$ is proven in \cite{bcan}.
That's why we suppose that $k\leq r-1$. For all natural $n$ we
define
$$
S_n=\beta_1^n+\beta_2^n+\cdots+\beta_r^n.
$$
We can rewrite the inequality \eqref{eq3.1} in following way
\begin{equation}\label{eq3.2}
S_k-S_{k+1}\leq\left(\frac{r-1}r\right)^k.
\end{equation}
Since
$$
\frac{r-1}r=\frac{S_1}r\leq\sqrt[k]{\frac{S_k}r}\leq
\sqrt[k+1]{\frac{S_{k+1}}r}\quad\text{(cf. \cite{HL, kh90})},
$$
we have
\begin{equation}\label{eq3.3}
S_{k+1}\geq\frac1{\sqrt[k]r}S_k^{\frac{k+1}k}
\end{equation}
and
\begin{equation}\label{eq3.4}
S_k\geq\frac{(r-1)^k}{r^{k-1}}.
\end{equation}
From \eqref{eq3.3} we see that
\begin{equation}\label{eq3.5}
S_k-S_{k+1}\leq S_k-\frac1{\sqrt[k]r}S_k^{\frac{k+1}k}.
\end{equation}
We consider the function
$$
f(x)=x-\frac1{\sqrt[k]r}x^{\frac{k+1}k},\quad x>0.
$$
According to \eqref{eq3.2} and \eqref{eq3.5} it is sufficient to prove that
$$
f(S_k)\leq\left(\frac{r-1}r\right)^k.
$$
From $f'(x)=1-\frac{k+1}{k\sqrt[k]r}x^{\frac1k}$, it follows that $f'(x)$ has unique
positive root
$$
x_0=\frac{rk^k}{(k+1)^k}
$$
and $f(x)$ decreases in $[x_0,\infty)$. According to \eqref{eq3.4},
$S_k\geq\frac{(r-1)^k}{r^{k-1}}$. Since $k\leq r-1$,
$\frac{(r-1)^k}{r^{k-1}}\geq x_0$. Therefore
$$
f(S_k)\leq f\left(\frac{(r-1)^k}{r^{k-1}}\right)=\left(\frac{r-1}r\right)^k.
\qedhere
$$
\end{proof}

\begin{thm}\label{thm3.2}
Let $G$ be an $n$-vertex graph and
$$
\V(G)=V_1\cup V_2\cup\dots\cup V_r,\quad
V_i\cap V_j=\emptyset,\quad i\neq j,
$$
where $V_i$ are $\delta_k$-small sets. Then for all
natural $k\leq r$ the following inequalities are satisfied
\begin{enumerate}[{\rm (i)}]
\item
$\D_k(G)\leq\frac{n(r-1)}r$;
\item
$r\geq\frac n{n-\D_k(G)}$.
\end{enumerate}
\end{thm}

\begin{proof}
Let $n_i=\abs{V_i}$, $i=1,2,\dots,r$. Then
$$
\sum_{v\in\V(G)}d^k(v)=\sum_{i=1}^r\sum_{v\in V_i}d^k(v)\leq
\sum_{i=1}^rn_i(n-n_i)^k.
$$
Let $\beta_i=1-\frac{n_i}n$, $i=1,2,\dots,r$. Then
$$
\sum_{v\in\V(G)}d^k(v)\leq
n^{k+1}\sum_{i=1}^r\beta_i(1-\beta_i)^k,\quad k\geq r.
$$
The inequality (i) follows from the last inequality and Lemma~\ref{lm3.1}.
Solving the inequality (i) for $r$, we derive the inequality (ii).
\end{proof}

% ----------------------------------------------------------------------
\section[Some corollaries from Theorem \ref{thm3.2}]
{Some corollaries from Theorem \ref{thm3.2}}

\begin{cor}\label{cor4.1}
Let $G$ be an $n$-vertex graph and let $k$ and $s$ be natural numbers such that
$k\leq\varphi^{(s)}(G)$. Then
\begin{enumerate}[{\rm (i)}]
\item
$\D_k(G)\leq\frac{\big(\varphi^{(s)}(G)-1\big)n}{\varphi^{(s)}(G)}\leq
\frac{\big(\varphi(G)-1\big)n}{\varphi(G)}\leq
\frac{(\omega(G)-1)n}{\omega(G)}\leq
\frac{(\chi(G)-1)n}{\chi(G)};
$
\item
$
\varphi^{(s)}(G)\geq\frac n{n-\D_k(G)}.
$
\end{enumerate}
\end{cor}

\begin{proof}
Let $\varphi^{(s)}(G)=r$ and $\V(G)=V_1\cup V_2\cup\dots\cup V_r$,
$V_i\cap V_j=\emptyset$, where $V_i$ are $\delta_k$-small sets.
Then the left inequality in (i) follows from Theorem~\ref{thm3.2}~{(i)}.
The other inequalities in (i) follow from the inequalities
$\varphi^{(s)}(G)\leq\varphi(G)\leq\omega(G)\leq\chi(G)$. The inequality (ii)
follows from Theorem~\ref{thm3.2}~{(ii)}.
\end{proof}

\begin{rem}\label{rem4.1}
In the case $k=s=1$, Corollary \ref{cor4.1} is proven in \cite{bcan}
(cf. Theorem 6.3 (i) and Theorem 6.2 (ii)).
\end{rem}

\begin{cor}\label{cor4.2}
Let $G$ be an $n$-vertex graph. Then for all natural $s\geq2$,
$$
\varphi^{(s)}(G)\geq\frac n{n-\D_2(G)}.
$$
\end{cor}

\begin{proof}
If $\varphi^{(2)}(G)=1$ then ${\rm E}(G)=\emptyset$, i.\,e.
$G=\overline{K}_n$ and the inequality is obvious.
If $\varphi^{(2)}(G)\geq2$ then $\varphi^{(s)}(G)\geq2$
because $s\geq2$. Therefore Corollary~\ref{cor4.2} follows
from Corollary~\ref{cor4.1}~{(ii)}.
\end{proof}

\begin{cor}[\cite{bn12}]\label{cor4.3}
For every $n$-vertex graph
$$
\varphi(G)\geq\frac n{n-\D_2(G)}.
$$
\end{cor}

\begin{proof}
This inequality follows from Corollary \ref{cor4.2} because
$\varphi^{(s)}(G)\leq\varphi(G)$.
\end{proof}

\begin{cor}[\cite{bcan}]\label{cor4.4}
Let $G$ be an $n$-vertex graph. Then for every natural $k\leq\varphi(G)$
$$
\varphi(G)\geq\frac n{n-\D_k(G)}.
$$
\end{cor}

\begin{proof}
According to Theorem \ref{thm2.1} there exists a natural number $s$ such that
$\varphi(G)=\varphi^{(s)}(G)$. Since $k\leq\varphi^{(s)}(G)$ from
Corollary~\ref{cor4.1}~{(ii)} we derive
$$
\varphi(G)=\varphi^{(s)}(G)\geq\frac n{n-\D_k(G)}.
$$
\end{proof}

\begin{cor}\label{cor4.5}
Let $G$ be an $n$-vertex graph. Then for every natural $s\geq3$
$$
\varphi^{(s)}(G)\geq\frac n{n-\D_3(G)}.
$$
\end{cor}

\begin{proof}
Since $s\geq3$, $\varphi^{(s)}(G)\geq\varphi^{(3)}(G)$. Therefore
it is sufficient to prove the inequality
\begin{equation}\label{eq4.1}
\varphi^{(3)}(G)\geq\frac n{n-\D_3(G)}.
\end{equation}
If $\varphi^{(3)}(G)\geq3$ then \eqref{eq4.1} follows from
Corollary~\ref{cor4.1}~{(ii)}. If $\varphi^{(3)}(G)=1$ then
the inequality \eqref{eq4.1} is obvious because $\D_3(G)=0$.
Let $\varphi^{(3)}(G)=2$ and $\V(G)=V_1\cup V_2$, where $V_i$,
$i=1$, 2 are $\delta_3$-small sets. Let $n_i=\abs{V_i}$, $i=1$, 2. Then
\begin{multline}\label{eq4.2}
\sum_{v\in\V(G)}d^3(v)=
\sum_{v\in V_1}d^3(v)+\sum_{v\in V_2}d^3(v)\leq\\
n_1(n-n_1)^3+n_2(n-n_2)^3=n_1n_2(n^2-2n_1n_2)\leq\frac{n^4}8.
\end{multline}
Therefore $\D_3(G)\leq\frac n2$ and we obtain
$$
\frac n{n-\D_3(G)}\leq2=\varphi^{(3)}(G).
$$
\end{proof}

Since $\varphi(G)\geq\varphi^{(3)}(G)$ from Corollary~\ref{cor4.5} we derive

\begin{cor}[\cite{bcan}]\label{cor4.6}
For every $n$-vertex graph $G$
$$
\varphi(G)\geq\frac n{n-\D_3(G)}.
$$
\end{cor}

\begin{cor}\label{cor4.7}
Let $G$ be an $n$-vertex graph and $\varphi^{(4)}(G)\neq2$.
Then for every natural $s\geq4$,
$$
\varphi^{(s)}(G)\geq\frac n{n-\D_4(G)}.
$$
\end{cor}

\begin{proof}
Since $\varphi^{(s)}(G)\geq\varphi^{(4)}(G)$ for $s\geq4$, it
sufficient to prove the inequality
\begin{equation}\label{eq4.3}
\varphi^{(4)}(G)\geq\frac n{n-\D_4(G)}.
\end{equation}

If $\varphi^{(4)}(G)\geq4$ the inequality \eqref{eq4.3} follows
from Corollary~\ref{cor4.1}~{(ii)}. If $\varphi^{(4)}(G)=1$ the
inequality \eqref{eq4.3} is obvious because $\D_4(G)=0$. It remains to
consider the case $\varphi^{(4)}(G)=3$.
Let $\V(G)=V_1\cup V_2\cup V_3$, where $V_i$,
are $\delta_4$-small sets and let $n_i=\abs{V_i}$, $i=1$, 2, 3. Then
\begin{multline}
\sum_{v\in\V(G)}d^4(v)=
\sum_{v\in V_1}d^4(v)+\sum_{v\in V_2}d^4(v)+\sum_{v\in V_3}d^4(v)\leq\\
n_1(n-n_1)^4+n_2(n-n_2)^4+n_3(n-n_3)^4.
\end{multline}
Denoting $\beta_i=1-\frac{n_i}n$, $i=1$, 2, 3 we receive
$$
\sum_{v\in\V(G)}d^4(v)\leq n^4\left(\sum_{i=1}^3(1-\beta_i)\beta_i^4\right).
$$
Since $\sum_{i=1}^3(1-\beta_i)\beta_i^4\leq\frac23$
(see the proof of Theorem~{5.4 (iii)} in \cite{bcan}) we take
$$
\D_4(G)\leq\frac23=\frac{\varphi^{(4)}(G)-1}{\varphi^{(4)}(G)}.
$$
Solving the last equation for $\varphi^{(4)}(G)$ we obtain \eqref{eq4.3}.
\end{proof}

\begin{cor}\label{cor4.8}
Let $G$ be an $n$-vertex graph and $\varphi^{(4)}(G)\neq2$. Then
\begin{equation}\label{eq4.3}
\varphi(G)\geq\frac n{n-\D_4(G)}.
\end{equation}
\end{cor}

\begin{rem}\label{rem4.2}
In \cite{bcan} it is proven that the inequlity \eqref{eq4.3} is held
if  $\varphi(G)\neq2$.
\end{rem}

% -----------------------------------------------------------

\section[Maximal $\delta_k$-sets]{Maximal $\delta_k$-sets}

We denote the maximal number of vertices in a $\delta_k$-set of $G$ by $\alpha^{(k)}(G)$.
$S(G)$ is the maximal number of vertices of small sets of $G$. From
Proposition~\ref{prop1.3} is easy to see that the next proposition holds.

\begin{prop}\label{prop5.1}
For every graph $G$
$$
\alpha^{(1)}(G)\geq\alpha^{(2)}(G)\geq\cdots\geq\alpha^{(k)}(G)\geq
\cdots\geq S(G)\geq\alpha(G).
$$
\end{prop}

\begin{rem}
Note that $\alpha^{(1)}(G)$ is denoted in \cite{bcan} by $S^\alpha(G)$.
\end{rem}

From Theorem \ref{thm2.1} we have

\begin{thm}\label{thm5.2}
For every graph $G$ there exists an unique number $k_0=k_0(G)$ such that
$$
\alpha^{(1)}(G)\geq\alpha^{(2)}(G)\geq\cdots\geq\alpha^{(k_0)}(G)=
\alpha^{(k_0+1)}(G)\cdots=S(G).
$$
\end{thm}

\begin{prop}\label{prop5.3}
Let $\V(G)=\set{v_1,v_2,\dots,v_n}$ and
$d(v_1)\leq d(v_2)\leq\cdots\leq d(v_n)$. Then
\begin{equation*}
\begin{split}
\alpha^{(k)}(G)&{}=\max\set{s\mid\D_k(\set{v_1,v_2,\dots v_s})\leq n-s}=\\
&{}=\max\set{s\mid\set{v_1,v_2,\dots v_s}\text{\rm\ is $\delta_k$-small set in $G$}}.\\
\end{split}
\end{equation*}
\end{prop}

\begin{proof}
Let $s_0=\max\set{s\mid\set{v_1,v_2,\dots v_s}\text{ is $\delta_k$-small set in $G$}}$.
Then $s_0\leq\alpha^{(k)}(G)$. Let $\alpha^{(k)}(G)=r$ and let
$\set{v_{i_1},v_{i_2},\dots,v_{i_r}}$ be a $\delta_k$-small set.
Since $\D_k(\set{v_1,v_2,\dots,v_r})\leq\D_k(\set{v_{i_1},v_{i_2},\dots,v_{i_r}})$
it follows that $\set{v_1,v_2,\dots,v_r}$ is $\delta_k$-small set too.
Therefore $\alpha^{(k)}(G)=r\leq s_0$.
\end{proof}

\begin{prop}\label{prop5.4}
For every natural $k$ are held the inequlities
$$
n-\Delta(G)\leq\alpha^{(k)}(G)\leq n-\delta(G).
$$
\end{prop}

\begin{proof}
The left inequality follows from the inequality $S(G)\geq n-\Delta(G)$
from \cite{bcan} and Proposition~\ref{prop5.1}. Let $r=\alpha^{(k)}(G)$.
According to Proposition~\ref{prop5.3}, $\set{v_1,v_2,\dots,v_r}$
is a $\delta_k$-small set. So
$$
\delta(G)=d(v_1)\leq\D_k(\set{v_1,v_2,\dots,v_r})\leq n-r=n-\alpha^{(k)}(G),
$$
hence $\alpha^{(k)}(G)\leq n-\delta(G)$.
\end{proof}

\begin{rem}\label{rem5.1}
The inequality $\alpha(G)\geq n-\Delta(G)$ is not always true.
For example, $\alpha(C_5)<5-\Delta(C_5)=3$.
\end{rem}

\begin{thm}\label{thm5.5}
Let $A\subseteq\V(G)$ be a $\delta_1$-small set of $G$ and
$s=\D_1(\V(G)\setminus A)$. Then
\begin{equation}\label{eq5.1}
\abs A\leq
\left\lfloor
\frac{n-s}2+\sqrt{\frac{(n-s)^2}4+ns-2e(G)}
\right\rfloor
\end{equation}
\end{thm}

\begin{proof}
$$
2e(G)=\sum_{v\in\V(G)}d(v)=\sum_{v\in A}d(v)+\sum_{v\in\V(G)\setminus A}d(v)
\leq\abs A(n-\abs A)+s(n-\abs A).
$$
Solving the derived quadric inequality for $\abs A$ we obtain the inequality~\ref{eq5.1}.
\end{proof}

\begin{cor}[\cite{bcan}]\label{cor5.6}
For every number $k$
\begin{equation}\label{eq5.2}
\begin{split}
\alpha^{(k)}(G)&{}\leq
\left\lfloor
\frac{n-\Delta(G)}2+\sqrt{\frac{(n-\Delta(G))^2}4+n\Delta(G)-2e(G)}
\right\rfloor\leq\\
&{}\leq\left\lfloor
\frac12+\sqrt{\frac14+n^2-n-2e(G)}
\right\rfloor.
\end{split}
\end{equation}
\end{cor}

\begin{proof}
According to Proposition \ref{prop5.1}, it is sufficient to prove \eqref{eq5.2}
only in the case $k=1$. Let $A$ be a maximal $\delta_1$-small set, i.\,e.
$\abs A=\alpha^{(1)}(G)$, and $s=\D_1\big(\V(G)\setminus A\big)$.
According to Theorem~\ref{thm5.5} the inequality \eqref{eq5.1} holds.
Since the right side of \eqref{eq5.1} is an increasing function for $s$
and $s\leq\Delta(G)\leq n-1$, the inequalities \eqref{eq5.2} follows from \eqref{eq5.1}.
\end{proof}

\section{$\alpha$-small sets}

\begin{defn}[\cite{bcan}]\label{defn6.1}
Let $G$ be an $n$-vertex graph and let $W\subseteq\V(G)$. We say that $W$ is
an $\alpha$-small set if
$$
\sum_{v\in W}\frac1{n-d(v)}\leq1.
$$
We denote the smallest natural number $r$ for which $\V(G)$ decomposes into $r$ $\alpha$-small
sets by $\varphi^\alpha(G)$.
\end{defn}

The idea for $\alpha$-small sets is coming from the following Caro-Wey inequality
(\cite{caro} and \cite{wei})
$$
\omega(G)\geq\sum_{v\in\V(G)}\frac1{n-d(v)}.
$$

We have the proposition

\begin{prop}[\cite{bcan}]\label{prop6.1}
$$
\varphi^{(1)}(G)\leq\varphi^\alpha(G)\leq\varphi(G).
$$
\end{prop}

The following problem is inspirited from Proposition \ref{prop6.1} and Theorem \ref{thm2.1}.

{\bf Problem.} Is it true that for every graph $G$ there exists natural number
$k_0=k_0(G)$ such that $\varphi^{(\alpha)}(G)=\varphi^{(k_0)}(G)$?

%\cite{kh90,hl34,bn12,caro,ee83,eg61,nen06,wei,west01,bcan}

%\bibliographystyle{amsplain}
%\bibliography{xbib}

\begin{thebibliography}{100}

\bibitem{bcan}
A.~Bojilov, Y.~Caro, A.~Hansberg, and N.~Nenov, \emph{Partitions of graphs into
  small and large sets}, 2012, arXiv:1205.1727.

\bibitem{bn12}
A.~Bojilov and N.~Nenov, \emph{An inequality for generalized chromatic graphs},
  Proceedings of the Forty First Spring Conference of Union of Bulgarian
  Mathematics (Borovets), Mathematics and education in mathematics, April 9--12
  2012, pp.~143--147.

\bibitem{caro}
Y.~Caro, \emph{New results on the independence number}, Tech. report, Tel-Aviv
  University, 1979.

\bibitem{hl34}
G.~H. Hardy, J.~F. Litelewood, and G.~Polya, \emph{Inequalities}, 1934.

\bibitem{kh90}
N.~Khadzhiivanov, \emph{Extremal theory of graphs}, Sofia University, Sofia,
  1990, (in Bulgarian).

\bibitem{nen06}
N.Nenov, \emph{Improvement of graph theory {W}ei's inequlity}, Proceedings of
  the Thirty Fifth Spring Conference of Union of Bulgarian Mathematics
  (Borovets), Mathematics and education in mathematics, April 5-8 2006,
  pp.~191--194.

\bibitem{wei}
V.~K. Wei, \emph{A lower bound on the stability number of a simple graph},
  Technical Memorandum 81--11217--9, Bell Laboratories, Murray Hill, NJ, 1981.

\bibitem{west01}
D.~B. West, \emph{Introduction to graph theory}, second ed., Prentice Hall,
  Inc., Upper Saddle River, NJ, 2001, xx+588 pp.

\end{thebibliography}

\end{document}